\documentclass{amsart}
\usepackage{amsmath}
\usepackage{amsfonts}

\newtheorem{theorem}{Theorem}
\theoremstyle{plain}

\newtheorem{corollary}{Corollary}

\newtheorem{definition}{Definition}
\newtheorem{example}{Example}

\newtheorem{remark}{Remark}

\numberwithin{equation}{section}
\input{tcilatex}

\begin{document}
\title[More on deviations from Mean Value ]{More on deviations from Mean
Value }
\author{Shoshana Abramovich}
\address{Department of Mathematics, University of Haifa, Haifa, Israel}
\email{abramos@math.haifa.ac.il}
\date{July 14, 2025 }

\begin{abstract}
This paper deals with more refinements of inequalities related to deviations
from Mean Value involving superquadratic and uniformly convex functions.
\end{abstract}

\maketitle

\section{\textbf{\ Introduction}}

This paper deals with more refinements of inequalities related to deviations
from Mean Value involving superquadratic and uniformly convex functions.

We start with quoting known definitions and theorems.

\begin{definition}
\label{Def1} \cite{AJS} A function $\varphi :I=\left[ 0,b\right) \rightarrow 
%TCIMACRO{\U{211d} }%
%BeginExpansion
\mathbb{R}
%EndExpansion
,$ $b\leq \infty $, is \textbf{superquadratic} provided that for all $x\in I$%
\ there exists a constant $C_{x}\in \ 
%TCIMACRO{\U{211d} }%
%BeginExpansion
\mathbb{R}
%EndExpansion
$ such that%
\begin{equation}
\varphi \left( y\right) \geq \varphi \left( x\right) +C_{x}\left( y-x\right)
+\varphi \left( \left\vert y-x\right\vert \right)  \label{1.1}
\end{equation}%
for all $y\in I$.\ If the reverse of (\ref{1.1}) holds then $\varphi $\ is
called subquadratic.
\end{definition}

\begin{remark}
\bigskip\ \label{Rem1} \cite{AJS} The function $\varphi $ is superquadratic
on $I=\left[ 0,b\right) \rightarrow 
%TCIMACRO{\U{211d} }%
%BeginExpansion
\mathbb{R}
%EndExpansion
$, if and only if the inequality$,$ 
\begin{equation*}
\varphi \left( \sum_{r=1}^{n}\lambda _{r}x_{r}\right) \leq
\sum_{r=1}^{n}\lambda _{r}\left( \varphi \left( x_{r}\right) -\varphi \left(
\left\vert x_{r}-\sum_{j=1}^{n}\lambda _{j}x_{j}\right\vert \right) \right)
\end{equation*}%
holds for $x_{r}\in I,$ $\lambda _{r}\geq 0,$ \ $r=1,...,n$, and $%
\sum_{r=1}^{n}\lambda _{r}=1.$
\end{remark}

\begin{definition}
\label{Def2} \cite{N} Let $\left[ a,b\right] \subset 
%TCIMACRO{\U{211d} }%
%BeginExpansion
\mathbb{R}
%EndExpansion
$ be an interval and $\Phi :\left[ 0,b-a\right] \rightarrow 
%TCIMACRO{\U{211d} }%
%BeginExpansion
\mathbb{R}
%EndExpansion
$ be a function. A function $f:\left[ a,b\right] \rightarrow 
%TCIMACRO{\U{211d} }%
%BeginExpansion
\mathbb{R}
%EndExpansion
$ is said\ to be generalized $\Phi $-uniformly convex\ if:%
\begin{eqnarray*}
tf\left( x\right) +\left( 1-t\right) f\left( y\right) &\geq &f\left(
tx+\left( 1-t\right) y\right) +t\left( 1-t\right) \Phi \left( \left\vert
x-y\right\vert \right) \\
\text{for \ }x,y &\in &\left[ a,b\right] \text{ \ and }t\in \left[ 0,1\right]
\text{.}
\end{eqnarray*}%
If in addition $\Phi \geq 0$, then $f$\ is said to be $\Phi $\textbf{%
-uniformly convex}\textit{, }or\textit{\ }\textbf{uniformly convex with
modulus }$\Phi $. In the special case where $\Phi \left( \left\vert
x-y\right\vert \right) =c\left( x-y\right) ^{2}$, $c>0$, $f$ is called 
\textbf{strongly convex} function.
\end{definition}

\begin{remark}
\label{Rem2} It is proved in \cite{Z} and \cite{N} that when $f$ is
uniformly convex on $\left[ a,b\right] $,\textbf{\ }there is always a
modulus $\Phi $ which is increasing and $\Phi \left( 0\right) =0$. In this
paper we use such a $\Phi $ when dealing with uniformly convex function.

It is also shown in \cite{Z} that the inequality%
\begin{equation*}
f\left( \sum_{r=1}^{n}\lambda _{r}x_{r}\right) \leq \sum_{r=1}^{n}\lambda
_{r}\left( f\left( x_{r}\right) -\Phi \left( \left\vert
x_{r}-\sum_{j=1}^{n}\lambda _{j}x_{j}\right\vert \right) \right)
\end{equation*}%
holds, when $x_{r}\in \left[ a,b\right] ,$ $0\leq \lambda _{r}\leq 1,$ $%
r=1,...,n,\ \sum_{r=1}^{n}\lambda _{r}=1$. It is easy to verify that this
inequality holds also for generalized uniformly convex functions and that in
these cases $\Phi \left( 0\right) \leq 0$\ .
\end{remark}

The following theorem is proved in \cite{CIPU}:

\begin{theorem}
\label{Th1} Let $n>1$\ be an integer and $x_{1},x_{2},...,x_{n}$\ be
positive real numbers. Denote $a=\frac{1}{n}\sum_{i=1}^{n}x_{i}$\ \ and \ $b=%
\frac{1}{n}\sum_{i=1}^{n}x_{i}^{2},$\ \ then $\ \underset{1\leq k\leq n}{%
\max }\left\{ \left\vert x_{k}-a\right\vert \right\} \leq \sqrt{\left(
n-1\right) \left( b-a^{2}\right) }$.
\end{theorem}

Theorem \ref{Th1} is a special case of the following Theorem \ref{Th2} for $%
p=2,$ $\ \alpha _{k}=\frac{1}{n},$ \ $k=1,...,n$.

\begin{theorem}
\label{Th2} \cite{ABP} Let $n>1$ be an integer and $x_{1},x_{2},...,x_{n}$
be positive real numbers. Denote \ $a=\sum_{i=1}^{n}\alpha _{i}x_{i}$ \ and
\ $c=\sum_{i=1}^{n}\alpha _{i}x_{i}^{p},$ \ where \ $0<\alpha _{i}<1,$ $\
i=1,...,n\ \ \ \sum \alpha _{i}=1,$ \ $p\geq 2,$ then \ \ 
\begin{equation*}
\ \underset{1\leq k\leq n}{\max }\left\{ \left\vert x_{k}-a\right\vert
\right\} \leq T\left( c-a^{p}\right) ^{\frac{1}{p}}
\end{equation*}%
where 
\begin{equation*}
T=\frac{\left( 1-\alpha _{0}\right) ^{1-\frac{1}{p}}}{\alpha _{0}^{\frac{1}{p%
}}\left( \alpha _{0}^{p-1}+\left( 1-\alpha _{0}\right) ^{p-1}\right) ^{\frac{%
1}{p}}},\qquad \alpha _{0}=\underset{1\leq k\leq n}{\min }\left( \alpha
_{k}\right) .
\end{equation*}
\end{theorem}

In \cite[Inequality (8)]{VK} the following is proved for Mean Value:

\begin{theorem}
\label{Th3}\textbf{\ }Let\textbf{\ }$y_{1}\leq y_{2}\leq ,...,y_{n}$ be $n$
real numbers, then for any positive integer $r$\ the inequality 
\begin{equation*}
y_{\left( k,j\right) }\leq \overline{y}+\left( \frac{n\left( n-j\right)
^{2r-1}g_{2r}}{j^{2r}+j\left( n-j\right) ^{2r-1}}\right) ^{\frac{1}{2r}}
\end{equation*}%
holds, for 
\begin{eqnarray*}
g_{2r} &=&\frac{1}{n}\sum_{i=1}^{n}\left\vert y_{i}-\overline{y}\right\vert
^{2r},\quad \overline{y}=\frac{1}{n}\sum_{i=1}^{n}y_{i},\quad y_{\left(
k,j\right) }=\sum_{i=k}^{j}\frac{y_{i}}{j-k+1}, \\
1 &\leq &k\leq j\leq n.
\end{eqnarray*}
\end{theorem}

In Section 2 after Teorems \ref{Th4} and \ref{Th5} we obtain in Theorem \ref%
{Th6} the main results where we extend Theorem \ref{Th3} for any real number 
$r\geq 1$ and for general positive coefficients. Then, in Theorem \ref{Th7}
we extend the results of Theorem \ref{Th6} to coefficients which are not
necessarily all nonnegative using Jensen-Steffensen inequality (see for
instance \cite[inequalities 1.1 and 1.2]{ABMP}).$\ \ $

\section{\textbf{Results: On deviations from Mean Values.}}

The new results proved in this section are theorems \ref{Th6} and \ref{Th7}.

For the benefit of the reader, we repeat theorems \ref{Th4} and \ref{Th5}
and their proofs which also appear in \cite[Theorems 12 and 13]{A}. Theorems %
\ref{Th4} and \ref{Th5} are special cases of Theorems \ref{Th6} and \ref{Th7}
but are proved in a different way..

\begin{theorem}
\label{Th4} Let $f$ be an uniformly convex function on the interval $\left[
a,b\right] $\ with modulus $\Phi \left( x\right) =mx^{p},$ $p\geq 1,$ $m>0$.
Let $n>1$ be an integer and $x_{i}\in \left[ a,b\right] ,$ $i=1,...,n$.
Denote \ $a=\sum_{i=1}^{n}\alpha _{i}x_{i}$ \ and \ $c=\sum_{i=1}^{n}\alpha
_{i}f\left( x_{i}\right) ,$ \ where \ $0<\alpha _{i}<1,$ $\ i=1,...,n\ \ \
\sum_{i=1}^{n}\alpha _{i}=1$. Then \ \ 
\begin{equation}
\ \underset{1\leq k\leq n}{\max }\left\{ \left\vert x_{k}-a\right\vert
\right\} \leq T\left( \sum_{i=1}^{n}\alpha _{i}\left\vert x_{i}-a\right\vert
\right) ^{\frac{1}{p}}\leq Tm^{-\frac{1}{p}}\left( c-f\left( a\right)
\right) ^{\frac{1}{p}}  \label{2.1}
\end{equation}%
where 
\begin{equation*}
T=\frac{\left( 1-\alpha _{0}\right) ^{1-\frac{1}{p}}}{\alpha _{0}^{\frac{1}{p%
}}\left( \alpha _{0}^{p-1}+\left( 1-\alpha _{0}\right) ^{p-1}\right) ^{\frac{%
1}{p}}},\qquad \alpha _{0}=\underset{1\leq k\leq n}{\min }\left( \alpha
_{k}\right) .
\end{equation*}
\end{theorem}

\begin{example}
\label{Ex1} Let $\varphi \left( x\right) ,$ $x\geq 0$ be a convex function
satisfying $\varphi ^{^{\prime }}\left( 0\right) >0$. It is easy to verify
that the function $g\left( x\right) =x\varphi \left( x\right) $ when $%
x_{i}\in 
%TCIMACRO{\U{211d} }%
%BeginExpansion
\mathbb{R}
%EndExpansion
_{+}$ and $0\leq \alpha _{i}\leq 1,$ $i=1,...,n,\ \sum_{i=1}^{n}\alpha
_{i}=1 $, satisfies the inequality%
\begin{equation*}
\sum_{i=1}^{n}\alpha _{i}g\left( x_{i}\right) -g\left( \sum_{j=1}^{n}\alpha
_{j}x_{j}\right) \geq \varphi ^{^{\prime }}\left( 0\right) \left(
\sum_{i=1}^{n}\alpha _{i}\left( x_{i}-\sum_{j=1}^{n}\alpha _{j}x_{j}\right)
^{2}\right) ,
\end{equation*}%
that is, we get that\ $g\left( x\right) $\ is strongly convex, and therefore
uniformly convex. Hence, it satisfies Theorem \ref{Th4}.
\end{example}

\begin{proof}
(of Theorem \ref{Th4}) \ As $\Phi \left( x\right) =mx^{p}$ \ $p\geq 1$, $%
m>0, $\ then applying Remark \ref{Rem2}%
\begin{equation}
\sum_{i=1}^{n}\alpha _{i}f\left( x_{i}\right) -f\left( \sum_{k=1}^{n}\alpha
_{k}x_{k}\right) \geq m\sum_{i=1}^{n}\alpha _{i}\left( \left\vert
x_{i}-a\right\vert \right) ^{p}  \label{2.2}
\end{equation}%
holds.

Denote $y_{i}=x_{i}-a,$ \ $i=1,...,n.$ \ Then from $\sum_{i=1}^{n}\alpha
_{i}y_{i}=0$ \ we get from H\"{o}lder's inequality 
\begin{eqnarray*}
\left( \alpha _{n}\left\vert y_{n}\right\vert \right) ^{p} &=&\left(
\left\vert \sum_{i=1}^{n-1}\alpha _{i}y_{i}\right\vert \right) ^{p}\leq
\left( \sum_{i=1}^{n-1}\alpha _{i}\left\vert y_{i}\right\vert \right) ^{p} \\
&=&\left( \sum_{i=1}^{n-1}\alpha _{i}^{1-\frac{1}{p}}\left( \alpha
_{i}\left\vert y_{i}\right\vert ^{p}\right) ^{\frac{1}{p}}\right) ^{p}\leq
\left( \sum_{i=1}^{n-1}\alpha _{i}\right) ^{p-1}\sum_{i=1}^{n-1}\alpha
_{i}\left\vert y_{i}\right\vert ^{p},
\end{eqnarray*}%
and 
\begin{equation*}
\left( \alpha _{n}\left\vert y_{n}\right\vert \right) ^{p}\leq \left(
1-\alpha _{n}\right) ^{p-1}\left( \sum_{i=1}^{n}\alpha _{i}\left\vert
y_{i}\right\vert ^{p}-\alpha _{n}\left\vert y_{n}\right\vert ^{p}\right)
\leq \left( 1-\alpha _{n}\right) ^{p-1}\sum_{i=1}^{n}\alpha _{i}\left\vert
y_{i}\right\vert ^{p}.
\end{equation*}%
Therefore from (\ref{2.2}) 
\begin{equation*}
\alpha _{n}\left( \alpha _{n}^{p-1}+\left( 1-\alpha _{n}\right)
^{p-1}\right) \left\vert y_{n}\right\vert ^{p}\leq \left( 1-\alpha
_{n}\right) ^{p-1}m^{-1}\left( c-f\left( a\right) \right)
\end{equation*}%
which by taking into consideration that $\frac{\left( 1-\alpha \right) ^{p-1}%
}{\alpha \left( \alpha ^{p-1}+\left( 1-\alpha \right) ^{p-1}\right) }$\ is
decreasing for $0<\alpha <1$ leads to (\ref{2.1}).
\end{proof}

In Theorem \ref{Th5} we deal with uniformly convex functions with modulus $%
\Phi $\ that satisfies special properties:

\begin{theorem}
\label{Th5} Let $f$ $:\left[ a,b\right) \rightarrow 
%TCIMACRO{\U{211d} }%
%BeginExpansion
\mathbb{R}
%EndExpansion
,$ be an uniformly convex function with modulus $\Phi $, which satisfies \ $%
\Phi \left( AB\right) \leq \Phi \left( A\right) \Phi \left( B\right) $ \ for 
$A>0$, $B>0$.

Let $x_{i}\in \left[ a,b\right) ,$ $i=1,...,n$. Denote $a=\sum_{i=1}^{n}%
\frac{x_{i}}{n}$ \ and \ $d=\frac{1}{n}\sum_{i=1}^{n}f\left( x_{i}\right) $.
If $\Phi $ is convex\ then 
\begin{eqnarray}
\underset{1\leq k\leq n}{\max }\left( \Phi \left( \left\vert
x_{k}-a\right\vert \right) \right) &\leq &\frac{\Phi \left( n-1\right) n}{%
n-1+\Phi \left( n-1\right) }\sum_{i=1}^{n}\frac{1}{n}\Phi \left( \left\vert
x_{i}-a\right\vert \right)  \label{2.3} \\
&\leq &\frac{\Phi \left( n-1\right) n}{n-1+\Phi \left( n-1\right) }\left(
d-f\left( a\right) \right) .  \notag
\end{eqnarray}
\end{theorem}

\begin{proof}
\ The function $f$ \ is uniformly convex with $\Phi $ defined on $\left[
a,b\right) ,$ therefore from Remark \ref{Rem2} for $\lambda _{i}=\frac{1}{n}%
, $ \ $i=1,...,n$%
\begin{equation}
\sum_{i=1}^{n}\frac{f\left( x_{i}\right) }{n}-f\left( \frac{%
\sum_{i=1}^{n}x_{i}}{n}\right) \geq \frac{1}{n}\sum_{i=1}^{n}\Phi \left(
\left\vert x_{i}-\frac{\sum_{i=1}^{n}x_{i}}{n}\right\vert \right)
\label{2.4}
\end{equation}%
holds.

In other words for $y_{i}=x_{i}-a,$ $i=1,...,n,$ 
\begin{equation}
\frac{1}{n}\sum_{i=1}^{n}\Phi \left( \left\vert y_{i}\right\vert \right)
\leq d-f\left( a\right)  \label{2.5}
\end{equation}%
holds.

From \ $\sum_{i=1}^{n}y_{i}=0$ \ we get that $\ \left\vert y_{n}\right\vert
=\left\vert -\sum_{i=1}^{n-1}y_{i}\right\vert .$ As $\Phi $ is positive, as
it is given that $\Phi $\ is convex too and according to Remark \ref{Rem2}, $%
\Phi $ is increasing and $\Phi \left( 0\right) =0$, therefore, \ \ 
\begin{eqnarray}
\Phi \left( \left\vert y_{n}\right\vert \right) &=&\Phi \left( \left\vert
-\sum_{i=1}^{n-1}y_{i}\right\vert \right) \leq \Phi \left(
\sum_{i=1}^{n-1}\left\vert y_{i}\right\vert \right)  \label{2.6} \\
&=&\Phi \left( \left( n-1\right) \sum_{i=1}^{n-1}\frac{\Phi ^{-1}\left( \Phi
\left( \left\vert y_{i}\right\vert \right) \right) }{n-1}\right)  \notag \\
&\leq &\Phi \left( \left( n-1\right) \Phi ^{-1}\left( \frac{%
\sum_{i=1}^{n-1}\Phi \left( \left\vert y_{i}\right\vert \right) }{n-1}%
\right) \right) .  \notag
\end{eqnarray}%
Indeed, the left side inequality is because $\Phi $ is increasing and the
right side inequality follows because $\Phi ^{-1}$ is concave and $\Phi $ $\ 
$is increasing.

As $\Phi $ satisfies also $\Phi \left( AB\right) \leq \Phi \left( A\right)
\Phi \left( B\right) $ we get that 
\begin{equation}
\Phi \left( \left( n-1\right) \Phi ^{-1}\left( \frac{\sum_{i=1}^{n-1}\Phi
\left( \left\vert y_{i}\right\vert \right) }{n-1}\right) \right) \leq \frac{%
\Phi \left( n-1\right) }{n-1}\sum_{i=1}^{n-1}\Phi \left( \left\vert
y_{i}\right\vert \right)  \label{2.7}
\end{equation}%
and from (\ref{2.6}) and (\ref{2.7}) we obtain 
\begin{equation*}
\Phi \left( \left\vert y_{n}\right\vert \right) \leq \frac{\Phi \left(
n-1\right) }{n-1}\left( \sum_{i=1}^{n-1}\Phi \left( \left\vert
y_{i}\right\vert \right) -\Phi \left( \left\vert y_{n}\right\vert \right)
\right) .
\end{equation*}%
From the last inequality as $\Phi (x)$ is positive and\ increasing, together
with (\ref{2.5}) 
\begin{equation*}
\frac{n-1+\Phi \left( n-1\right) }{n-1}\Phi \left( \left\vert
y_{n}\right\vert \right) \leq \frac{\Phi \left( n-1\right) }{n-1}%
\sum_{i=1}^{n}\Phi \left( \left\vert y_{i}\right\vert \right) \leq \frac{%
\Phi \left( n-1\right) }{n-1}n\left( d-f\left( a\right) \right)
\end{equation*}%
holds, which is equivalent to (\ref{2.3}).
\end{proof}

Up to now we dealt with evaluating the $\max \left\vert
x_{k}-\sum_{i=1}^{n}t_{i}x_{i}\right\vert ,$ $k=1,...,n$. Now we evaluate $%
\left\vert x_{k,j}-\overline{x}\right\vert $. \ 

Applying the technique used in the proof of Theorem \ref{Th3} we get the
following theorem which deals with more general cases than in Theorem \ref%
{Th3} and refers also to the superquadratic functions $x^{p},$ $p\geq 2,$ $%
x\geq 0$ and to uniformly convex functions with modulus $\Phi =x^{p},$ $%
p\geq 2,$ $x\geq 0$, and general positive coefficients whereas Theorem \ref%
{Th3} deals only with even $p$ and fixed coefficients. Moreover, in Theorem %
\ref{Th7} we show that we can relax even more the conditions on the
coefficients.

\begin{theorem}
\label{Th6} Let $y_{i}$ $i=1,...,n$\ be real numbers and $%
\sum_{i=1}^{n}t_{i}y_{i}=0$. Let $t_{i}>0,$ $i=1,...,n$ satisfy $%
\sum_{i=1}^{n}t_{i}=1$, and let $r\geq 1$ be a real number, then:

A)%
\begin{eqnarray}
&&  \label{2.8} \\
\left\vert \sum_{i=k}^{j}\frac{t_{i}y_{i}}{\sum_{v=k}^{j}t_{v}}\right\vert
&\leq &\left( \frac{\sum_{i=1}^{n}t_{i}\left\vert y_{i}\right\vert ^{2r}}{%
\left( \sum_{i=k}^{j}t_{i}\right) +\left( \sum_{i=k}^{j}t_{i}\right)
^{2r}\left( 1-\sum_{i=k}^{j}t_{i}\right) ^{1-2r}}\right) ^{\frac{1}{2r}}. 
\notag
\end{eqnarray}%
B)\ \ If in addition $y_{i}=x_{i}-\overline{x},$ $i=1,...,n$,\ and $%
\overline{x}=\sum_{i=1}^{n}t_{i}x_{i}$,\ then .%
\begin{equation}
\left\vert x_{k,j}-\overline{x}\right\vert \leq \left( \frac{%
\sum_{i=1}^{n}t_{i}\left\vert x_{i}-\overline{x}\right\vert ^{2r}}{\left(
\sum_{i=k}^{j}t_{i}\right) +\left( \sum_{i=k}^{j}t_{i}\right) ^{2r}\left(
1-\sum_{i=k}^{j}t_{i}\right) ^{1-2r}}\right) ^{\frac{1}{2r}},  \label{2.9}
\end{equation}%
where $x_{\left( k,j\right) }=\sum_{i=k}^{j}\frac{x_{i}}{\sum_{v=k}^{j}t_{v}}
$.

C) \ If also $x_{1}\geq x_{2}\geq ,...,\geq x_{n}$, it follows that $x_{1,j}$
is decreasing with $j,$ $j=1,...,n$, when $t_{i}=\frac{1}{n},$ $i=1,...,n,$ $%
\overline{x}=\frac{1}{n}\sum_{i=1}^{n}x_{i}$, and then 
\begin{equation}
x_{1,j}\leq \overline{x}+\left( \frac{\sum_{i=1}^{n}\frac{1}{n}\left\vert
x_{i}-\overline{x}\right\vert ^{2r}}{\left( \sum_{i=1}^{j}\frac{1}{n}\right)
+\left( \sum_{i=1}^{j}\frac{1}{n}\right) ^{2r}\left( 1-\sum_{i=1}^{j}\frac{1%
}{n}\right) ^{1-2r}}\right) ^{\frac{1}{2r}}.  \label{2.10}
\end{equation}%
holds.
\end{theorem}

\begin{proof}
\bigskip

\begin{equation*}
\sum_{i=1}^{n}t_{i}\left\vert y_{i}\right\vert
^{2r}=\sum_{i=1}^{n}t_{i}\left( \left\vert y_{i}\right\vert ^{2}\right)
^{r}\qquad \qquad \qquad \qquad \qquad \qquad \qquad
\end{equation*}%
\begin{equation*}
=\left( \sum_{i=k}^{j}t_{i}\right) \left( \sum_{i=k}^{j}\frac{t_{i}\left(
\left\vert y_{i}\right\vert ^{2}\right) ^{r}}{\sum_{v=k}^{j}t_{v}}\right)
+\left( 1-\sum_{i=k}^{j}t_{i}\right) \left( \sum_{i=1,i\neq k,...,j}^{n}%
\frac{t_{i}\left( \left\vert y_{i}\right\vert ^{2}\right) ^{r}}{%
1-\sum_{v=k}^{j}t_{v}}\right)
\end{equation*}%
\begin{equation*}
\geq \left( \sum_{i=k}^{j}t_{i}\right) \left( \sum_{i=k}^{j}\frac{%
t_{i}\left\vert y_{i}\right\vert ^{2}}{\sum_{v=k}^{j}t_{v}}\right)
^{r}+\left( 1-\sum_{i=k}^{j}t_{i}\right) \left( \sum_{i=1,i\neq k,...,j}^{n}%
\frac{t_{i}\left\vert y_{i}\right\vert ^{2}}{1-\sum_{v=k}^{j}t_{v}}\right)
^{r}
\end{equation*}%
\begin{equation*}
\geq \left( \sum_{i=k}^{j}t_{i}\right) \left( \sum_{i=k}^{j}\left( \frac{%
t_{i}y_{i}}{\sum_{v=k}^{j}t_{v}}\right) ^{2}\right) ^{r}+\left(
1-\sum_{i=k}^{j}t_{i}\right) \left( \sum_{i=1,i\neq k,...,j}^{n}\left( \frac{%
t_{i}y_{i}}{1-\sum_{v=k}^{j}t_{v}}\right) ^{2}\right) ^{r}
\end{equation*}%
\begin{equation*}
=\left( \sum_{i=k}^{j}t_{i}\right) \left( \left( \sum_{i=k}^{j}\frac{%
t_{i}y_{i}}{\sum_{v=k}^{j}t_{v}}\right) ^{2}\right) ^{r}+\left(
1-\sum_{i=k}^{j}t_{i}\right) \left( \left( -\sum_{i=k}^{j}\frac{t_{i}y_{i}}{%
1-\sum_{v=k}^{j}t_{v}}\right) ^{2}\right) ^{r}
\end{equation*}%
\begin{equation*}
=\left( \sum_{i=k}^{j}t_{i}\right) \left( \left( \sum_{i=k}^{j}\frac{%
t_{i}y_{i}}{\sum_{v=k}^{j}t_{v}}\right) ^{2}\right) ^{r}\qquad \qquad \qquad
\qquad \qquad \qquad
\end{equation*}%
\begin{eqnarray*}
&&+\left( 1-\sum_{i=k}^{j}t_{i}\right) ^{1-2r}\left(
\sum_{i=k}^{j}t_{i}\right) ^{2r}\left( \left( \sum_{i=k}^{j}\frac{t_{i}y_{i}%
}{\sum_{v=k}^{j}t_{v}}\right) ^{2}\right) ^{r} \\
&=&\left( \left( \sum_{i=k}^{j}\frac{t_{i}y_{i}}{\sum_{v=k}^{j}t_{i}}\right)
^{2}\right) ^{r}\left[ \left( \sum_{i=k}^{j}t_{i}\right) +\left(
\sum_{i=k}^{j}t_{i}\right) ^{2r}\left( 1-\sum_{i=k}^{j}t_{i}\right) ^{1-2r}%
\right] ,
\end{eqnarray*}%
The first inequality results from convexity of $x^{r},$ $r\geq 1,$ $x\geq 0$%
, and the second inequality from the convexity of $x^{2}$.

Hence, 
\begin{eqnarray*}
&&\sum_{i=1}^{n}t_{i}\left\vert y_{i}\right\vert ^{2r} \\
&\geq &\left\vert \sum_{i=k}^{j}\frac{t_{i}y_{i}}{\sum_{v=k}^{j}t_{v}}%
\right\vert ^{2r}\left[ \left( \sum_{i=k}^{j}t_{i}\right) +\left(
\sum_{i=k}^{j}t_{i}\right) ^{2r}\left( 1-\sum_{i=k}^{j}t_{i}\right) ^{1-2r}%
\right] ,
\end{eqnarray*}%
that is, 
\begin{equation*}
\left\vert \sum_{i=k}^{j}\frac{t_{i}y_{i}}{\sum_{v=k}^{j}t_{v}}\right\vert
^{2r}\leq \frac{\sum_{i=1}^{n}t_{i}\left\vert y_{i}\right\vert ^{2r}}{\left[
\left( \sum_{i=k}^{j}t_{i}\right) +\left( \sum_{i=k}^{j}t_{i}\right)
^{2r}\left( 1-\sum_{i=k}^{j}t_{i}\right) ^{1-2r}\right] },
\end{equation*}%
from which (\ref{2.8}) follows, and the proof of case A is complete.

If in addition $y_{i}=x_{i}-\overline{x},$ $i=1,...,n$, then\ 
\begin{equation*}
\left\vert \sum_{i=k}^{j}\frac{t_{i}\left( x_{i}-\overline{x}\right) }{%
\sum_{v=k}^{j}t_{v}}\right\vert ^{2r}\leq \frac{\sum_{i=1}^{n}t_{i}\left%
\vert x_{i}-\overline{x}\right\vert ^{2r}}{\left[ \left(
\sum_{i=k}^{j}t_{i}\right) +\left( \sum_{i=k}^{j}t_{i}\right) ^{2r}\left(
1-\sum_{i=k}^{j}t_{i}\right) ^{1-2r}\right] }.
\end{equation*}%
\ and denoting 
\begin{equation*}
\sum_{i=k}^{j}\left( \frac{t_{i}\left( x_{i}\right) }{\sum_{v=k}^{j}t_{v}}%
\right) =x_{k,j}
\end{equation*}%
we get that%
\begin{equation*}
\left\vert x_{k,j}-\overline{x}\right\vert ^{2r}\leq \frac{%
\sum_{i=1}^{n}t_{i}\left\vert x_{i}-\overline{x}\right\vert ^{2r}}{\left[
\left( \sum_{i=k}^{j}t_{i}\right) +\left( \sum_{i=k}^{j}t_{i}\right)
^{2r}\left( 1-\sum_{i=k}^{j}t_{i}\right) ^{1-2r}\right] },
\end{equation*}%
from which (\ref{2.9}) follows, the proof of case B is complete.

If $x_{1}\geq x_{2},\geq ,...,\geq x_{n}$\ then, it is easy to verify that
when $t_{i}=\frac{1}{n},$ $i=1,..,n,$ $x_{1,j}$is decreasing, $x_{1,j}-%
\overline{x}\geq 0$ and (\ref{2.10}) holds, and case C) is proved and the
proof of the theorem is complete.
\end{proof}

Using Theorem \ref{Th6} it follows from Remark \ref{Rem1} regarding
superquadracity and Remark \ref{Rem2} regarding uniform convexity that:

\begin{corollary}
\label{Cor1} \textbf{a}. If $f\ $is the superquadratic funnction $f\left(
x\right) =x^{2r},$ $x\in \left[ 0,b\right) ,$ and $r\geq 1,$\ where $%
\overline{x}=\sum_{i=1}^{n}t_{i}x_{i}$, $x_{\left( k,j\right)
}=\sum_{i=k}^{j}\frac{x_{i}}{\sum_{v=k}^{j}t_{v}},$ $0\leq t_{i}\leq 1,$\ $%
\sum_{i=1}^{n}t_{i}=1$,\ then the inequalities%
\begin{eqnarray*}
\left\vert x_{k,j}-\overline{x}\right\vert &\leq &\left( \frac{%
\sum_{i=1}^{n}t_{i}\left\vert x_{i}-\overline{x}\right\vert ^{2r}}{\left(
\sum_{i=k}^{j}t_{i}\right) +\left( \sum_{i=k}^{j}t_{i}\right) ^{2r}\left(
1-\sum_{i=k}^{j}t_{i}\right) ^{1-2r}}\right) ^{\frac{1}{2r}} \\
&\leq &\left( \frac{\sum_{i=1}^{n}t_{i}x_{i}^{2r}-\left(
\sum_{i=1}^{n}t_{i}x_{i}\right) ^{2r}}{\left( \sum_{i=k}^{j}t_{i}\right)
+\left( \sum_{i=k}^{j}t_{i}\right) ^{2r}\left( 1-\sum_{i=k}^{j}t_{i}\right)
^{1-2r}}\right) ^{\frac{1}{2r}},
\end{eqnarray*}%
and if also $x_{1}\geq x_{2}\geq ,...,\geq x_{n}$, it follows that $x_{1,j}$
is decreasing with $j,$ $j=1,...,n$, when $t_{i}=\frac{1}{n},$ $i=1,...,n,$ $%
\overline{x}=\frac{1}{n}\sum_{i=1}^{n}x_{i}$, and then the inequalities 
\begin{eqnarray*}
x_{1,j} &\leq &\overline{x}+\left( \frac{\sum_{i=1}^{n}\frac{1}{n}\left\vert
x_{i}-\overline{x}\right\vert ^{2r}}{\left( \sum_{i=1}^{j}\frac{1}{n}\right)
+\left( \sum_{i=1}^{j}\frac{1}{n}\right) ^{2r}\left( 1-\sum_{i=1}^{j}\frac{1%
}{n}\right) ^{1-2r}}\right) ^{\frac{1}{2r}} \\
&\leq &\overline{x}+\left( \frac{\sum_{i=1}^{n}\frac{1}{n}x_{i}^{2r}-\left(
\sum_{i=1}^{n}\frac{1}{n}x_{i}\right) ^{2r}}{\left( \sum_{i=1}^{j}\frac{1}{n}%
\right) +\left( \sum_{i=1}^{j}\frac{1}{n}\right) ^{2r}\left( 1-\sum_{i=1}^{j}%
\frac{1}{n}\right) ^{1-2r}}\right) ^{\frac{1}{2r}}.
\end{eqnarray*}%
hold for the superquadratic function $f\left( x\right) =x^{2r},$ $r\geq 1$.

\textbf{b.} \ If $f:\left[ a,b\right) \rightarrow 
%TCIMACRO{\U{211d} }%
%BeginExpansion
\mathbb{R}
%EndExpansion
$ is uniformly convex with modulus $\Phi =x^{2r}$, $x\in \left[ 0,b-a\right)
,\ r\geq 1$, then we get the inequalities%
\begin{eqnarray*}
\left\vert x_{k,j}-\overline{x}\right\vert &\leq &\left( \frac{%
\sum_{i=1}^{n}t_{i}\left\vert x_{i}-\overline{x}\right\vert ^{2r}}{\left(
\sum_{i=k}^{j}t_{i}\right) +\left( \sum_{i=k}^{j}t_{i}\right) ^{2r}\left(
1-\sum_{i=k}^{j}t_{i}\right) ^{1-2r}}\right) ^{\frac{1}{2r}} \\
&\leq &\left( \frac{\sum_{i=1}^{n}t_{i}f\left( x_{i}\right) -f\left(
\sum_{i=1}^{n}t_{i}x_{i}\right) }{\left( \sum_{i=k}^{j}t_{i}\right) +\left(
\sum_{i=k}^{j}t_{i}\right) ^{2r}\left( 1-\sum_{i=k}^{j}t_{i}\right) ^{1-2r}}%
\right) ^{\frac{1}{2r}},
\end{eqnarray*}%
and if also $x_{1}\geq x_{2}\geq ,...,\geq x_{n}$, it follows that $x_{1,j}$
is decreasing with $j,$ $j=1,...,n$, when $t_{i}=\frac{1}{n},$ $i=1,...,n,$ $%
\overline{x}=\frac{1}{n}\sum_{i=1}^{n}x_{i}$, and then the inequalities 
\begin{eqnarray*}
x_{1,j} &\leq &\overline{x}+\left( \frac{\sum_{i=1}^{n}\frac{1}{n}\left\vert
x_{i}-\overline{x}\right\vert ^{2r}}{\left( \sum_{i=1}^{j}\frac{1}{n}\right)
+\left( \sum_{i=1}^{j}\frac{1}{n}\right) ^{2r}\left( 1-\sum_{i=1}^{j}\frac{1%
}{n}\right) ^{1-2r}}\right) ^{\frac{1}{2r}} \\
&\leq &\overline{x}+\left( \frac{\sum_{i=1}^{n}\frac{1}{n}f\left(
x_{i}\right) -f\left( \sum_{i=1}^{n}\frac{1}{n}x_{i}\right) }{\left(
\sum_{i=1}^{j}\frac{1}{n}\right) +\left( \sum_{i=1}^{j}\frac{1}{n}\right)
^{2r}\left( 1-\sum_{i=1}^{j}\frac{1}{n}\right) ^{1-2r}}\right) ^{\frac{1}{2r}%
}.
\end{eqnarray*}
\end{corollary}

In the next theorem we prove using Jensen-Steffensen inequality, that we can
relax the conditions on the coefficients $t_{i},$ $i=1,...,n$ to get results
similar to Theorem \ref{Th6}. In this case the conditions on $t_{i}$ are:%
\begin{eqnarray}
0 &\leq &\sum_{i=1}^{j}t_{i}\leq \sum_{i=1}^{n}t_{i}=1,\quad j=1,...,n\quad
\label{2.11} \\
0 &\leq &\sum_{i=1}^{j}t_{i}\leq \sum_{i=1}^{k}t_{i},\quad j\leq k,\quad 
\notag \\
0 &<&t_{k+1},\quad 0\leq \sum_{i=k+1}^{l}t_{i}\leq \sum_{i=k+1}^{n}t_{i}, 
\notag
\end{eqnarray}%
and $y_{1}\leq y_{2},...,\leq y_{n}$.

\bigskip Jensen-Steffensen's inequality states that (see for instance \cite[%
inequalities 1.1 and 1.2]{ABMP}) if \ $\varphi :I\rightarrow \mathbb{R}$\ \
is convex, then 
\begin{equation}
\varphi \left( \frac{1}{P_{n}}\sum_{i=1}^{n}\rho _{i}\zeta _{i}\right) \leq 
\frac{1}{P_{n}}\sum_{i=1}^{n}\rho _{i}\varphi \left( \zeta _{i}\right)
\label{2.12}
\end{equation}%
holds, where \ $I$\ \ is an interval in \ $\mathbb{R},\ \boldsymbol{\zeta }%
=\left( \zeta _{1,...,}\zeta _{n}\right) $\ \ is any monotonic $n$-tuple in
\ $I^{n}$\ \ and \ $\boldsymbol{\rho }=\left( \rho _{1},...,\rho _{n}\right) 
$\ \ is a real $n$-tuple that satisfies

\begin{align}
0& \leq P_{j}\leq P_{n}\;,\quad j=1,...,n\;,\quad P_{n}>0\;,  \label{2.13} \\
P_{j}& =\sum_{i=1}^{j}\rho _{i}\;,\quad \overline{P}_{j}=\sum_{i=j}^{n}\rho
_{i}\;,\quad j=1,...,n\text{ }.  \notag
\end{align}

\begin{theorem}
\label{Th7} Let $y_{i},$ $i=1,...,n$\ be real numbers, $y_{1}\leq y_{2}\leq
,...,\leq y_{n}$, and $\sum_{i=1}^{n}t_{i}y_{i}=0$. Let $t_{i},$ $i=1,...,n$
satisfy (\ref{2.11}), and let $r\geq 1$ be a real number, then the
inequality:%
\begin{eqnarray}
&&  \label{2.14} \\
\left\vert \sum_{i=1}^{k}\frac{t_{i}y_{i}}{\sum_{v=1}^{k}t_{v}}\right\vert 
&\leq &\left( \frac{\sum_{i=1}^{n}t_{i}\left\vert y_{i}\right\vert ^{2r}}{%
\left( \sum_{i=1}^{k}t_{i}\right) +\left( \sum_{i=1}^{k}t_{i}\right)
^{2r}\left( \sum_{i=k+1}^{n}t_{i}\right) ^{1-2r}}\right) ^{\frac{1}{2r}}. 
\notag
\end{eqnarray}%
holds.
\end{theorem}

\begin{proof}
\bigskip

\begin{equation*}
\sum_{i=1}^{n}t_{i}\left\vert y_{i}\right\vert
^{2r}=\sum_{i=1}^{n}t_{i}\left( \left\vert y_{i}\right\vert ^{2}\right)
^{r}=\sum_{i=1}^{n}t_{i}\left( y_{i}^{2}\right) ^{r}\qquad \qquad \qquad
\qquad \qquad \qquad \qquad
\end{equation*}%
\begin{equation*}
=\left( \sum_{i=1}^{k}t_{i}\right) \left( \sum_{i=1}^{k}\frac{t_{i}\left(
y_{i}^{2}\right) ^{r}}{\sum_{v=1}^{k}t_{v}}\right) +\left(
\sum_{i=k+1}^{n}t_{i}\right) \left( \sum_{i=k+1}^{n}\frac{t_{i}\left(
y_{i}^{2}\right) ^{r}}{\sum_{v=k+1}^{n}t_{v}}\right)
\end{equation*}%
\begin{equation*}
\geq \left( \sum_{i=1}^{k}t_{i}\right) \left( \sum_{i=1}^{k}\frac{%
t_{i}y_{i}^{2}}{\sum_{v=1}^{k}t_{v}}\right) ^{r}+\left(
\sum_{i=k+1}^{n}t_{i}\right) \left( \sum_{i=k+1}^{n}\frac{t_{i}y_{i}^{2}}{%
\sum_{v=k+1}^{n}t_{v}}\right) ^{r}
\end{equation*}%
\begin{equation*}
\geq \left( \sum_{i=1}^{k}t_{i}\right) \left( \sum_{i=1}^{k}\left( \frac{%
t_{i}y_{i}}{\sum_{v=1}^{k}t_{v}}\right) ^{2}\right) ^{r}+\left(
\sum_{i=k+1}^{n}t_{i}\right) \left( \sum_{i=k+1}^{n}\left( \frac{t_{i}y_{i}}{%
\sum_{v=k+1}^{n}t_{v}}\right) ^{2}\right) ^{r}
\end{equation*}%
\begin{equation*}
=\left( \sum_{i=1}^{k}t_{i}\right) \left( \left( \sum_{i=1}^{k}\frac{%
t_{i}y_{i}}{\sum_{v=1}^{k}t_{v}}\right) ^{2}\right) ^{r}+\left(
\sum_{i=k+1}^{n}t_{i}\right) \left( \left( -\sum_{i=k+1}^{n}\frac{t_{i}y_{i}%
}{\sum_{v=k+1}^{n}t_{v}}\right) ^{2}\right) ^{r}
\end{equation*}%
\begin{equation*}
=\left( \sum_{i=1}^{k}t_{i}\right) \left( \left( \sum_{i=1}^{k}\frac{%
t_{i}y_{i}}{\sum_{v=1}^{k}t_{v}}\right) ^{2}\right) ^{r}\qquad \qquad \qquad
\qquad \qquad \qquad
\end{equation*}%
\begin{eqnarray*}
&&+\left( \sum_{i=k+1}^{n}t_{i}\right) ^{1-2r}\left(
\sum_{i=1}^{k}t_{i}\right) ^{2r}\left( \left( \sum_{i=1}^{k}\frac{t_{i}y_{i}%
}{\sum_{v=1}^{k}t_{v}}\right) ^{2}\right) ^{r} \\
&=&\left( \left( \sum_{i=1}^{k}\frac{t_{i}y_{i}}{\sum_{v=1}^{k}t_{i}}\right)
^{2}\right) ^{r}\left[ \left( \sum_{i=1}^{k}t_{i}\right) +\left(
\sum_{i=1}^{k}t_{i}\right) ^{2r}\left( \sum_{i=k+1}^{n}t_{i}\right) ^{1-2r}%
\right] ,
\end{eqnarray*}%
The first inequality results from convexity of $x^{r},$ $r\geq 1,$ $x\geq 0$%
, and the second inequality from the convexity of $x^{2}$. In both
inequalities we took into consideration the special property of the number $%
k $ as appears in (\ref{2.11}), and Jensen-Steffensen inequality with
coefficients satisfying (\ref{2.13}) - what is called Jensen-Steffensen's
type coefficients.

Hence, 
\begin{eqnarray*}
&&\sum_{i=1}^{n}t_{i}\left\vert y_{i}\right\vert ^{2r} \\
&\geq &\left\vert \sum_{i=1}^{k}\frac{t_{i}y_{i}}{\sum_{v=1}^{k}t_{v}}%
\right\vert ^{2r}\left[ \left( \sum_{i=1}^{k}t_{i}\right) +\left(
\sum_{i=1}^{k}t_{i}\right) ^{2r}\left( \sum_{i=k+1}^{n}t_{i}\right) ^{1-2r}%
\right] ,
\end{eqnarray*}%
that is, 
\begin{equation*}
\left\vert \sum_{i=1}^{k}\frac{t_{i}y_{i}}{\sum_{v=1}^{k}t_{v}}\right\vert
^{2r}\leq \frac{\sum_{i=1}^{n}t_{i}\left\vert y_{i}\right\vert ^{2r}}{\left[
\left( \sum_{i=1}^{k}t_{i}\right) +\left( \sum_{i=1}^{k}t_{i}\right)
^{2r}\left( \sum_{i=k+1}^{n}t_{i}\right) ^{1-2r}\right] },
\end{equation*}%
from which (\ref{2.14}) follows, and the proof of the theorem is complete.
\end{proof}

Jensen-Steffensen's inequality states that (see for instance \cite[Theorem 1]%
{ABMP}) if \ $\varphi :I\rightarrow \mathbb{R}$\ \ is superquadratic and $%
\varphi ^{^{\prime }}$ is superadditive, then 
\begin{equation}
\frac{1}{P_{n}}\sum_{i=1}^{n}\rho _{i}\varphi \left( \zeta _{i}\right) \geq
\varphi \left( \frac{1}{P_{n}}\sum_{i=1}^{n}\rho _{i}\zeta _{i}\right) +%
\frac{1}{P_{n}}\sum_{i=1}^{n}\rho _{i}\varphi \left( \left\vert \zeta _{i}-%
\frac{1}{P_{n}}\sum_{i=1}^{n}\rho _{i}\varphi \left( \zeta _{i}\right)
\right\vert \right)  \label{2.15}
\end{equation}%
holds, where \ $I$\ \ is an interval in \ $\mathbb{R}_{+},\ \boldsymbol{%
\zeta }=\left( \zeta _{1,...,}\zeta _{n}\right) $\ \ is any monotonic $n$%
-tuple in \ $I^{n}$\ \ and \ $\boldsymbol{\rho }=\left( \rho _{1},...,\rho
_{n}\right) $\ \ is a real $n$-tuple that satisfies

\begin{align}
0& \leq P_{j}\leq P_{n}\;,\quad j=1,...,n\;,\quad P_{n}>0\;,  \label{2.16} \\
P_{j}& =\sum_{i=1}^{j}\rho _{i}\;,\quad \overline{P}_{j}=\sum_{i=j}^{n}\rho
_{i}\;,\quad j=1,...,n\text{ }.  \notag
\end{align}

\begin{corollary}
\label{Cor2} \textbf{a}. If $f\ $is the superquadratic funnction $f\left(
x\right) =x^{2r},$ $x\in \left[ 0,b\right) ,$ and $r\geq 1,$\ where $%
\overline{x}=\sum_{i=1}^{n}t_{i}x_{i}$, $x_{1,k}=\sum_{i=1}^{k}\frac{x_{i}}{%
\sum_{\nu =1}^{k}t_{v}},$ $t_{i}$ satisfies (\ref{2.11})$,$\ $%
\sum_{i=1}^{n}t_{i}=1$,\ $y_{i}=x_{i}-\overline{x},$ $i=1,...n$,\ then it
follows that%
\begin{eqnarray*}
\left\vert x_{1,k}-\overline{x}\right\vert  &=&\left\vert \sum_{i=1}^{k}%
\frac{t_{i}y_{i}}{\sum_{v=1}^{k}t_{v}}\right\vert  \\
&\leq &\left( \frac{\sum_{i=1}^{n}t_{i}\left\vert y_{i}\right\vert ^{2r}}{%
\left( \sum_{i=1}^{k}t_{i}\right) +\left( \sum_{i=1}^{k}t_{i}\right)
^{2r}\left( \sum_{i=k+1}^{n}t_{i}\right) ^{1-2r}}\right) ^{\frac{1}{2r}} \\
. &\leq &\left( \frac{\sum_{i=1}^{n}t_{i}x_{i}^{2r}-\left(
\sum_{i=1}^{n}t_{i}x_{i}\right) ^{2r}}{\left( \sum_{i=1}^{k}t_{i}\right)
+\left( \sum_{i=1}^{k}t_{i}\right) ^{2r}\left( \sum_{i=k+1}^{n}t_{i}\right)
^{1-2r}}\right) ^{\frac{1}{2r}}
\end{eqnarray*}

\textbf{b.} \ If $f:\left[ a,b\right) \rightarrow 
%TCIMACRO{\U{211d} }%
%BeginExpansion
\mathbb{R}
%EndExpansion
$ is uniformly convex with modulus $\Phi \left( x\right) =x^{2r}$, satisfing
the inequality $f^{^{\prime }}\left( y\right) -f^{^{\prime }}\left( x\right)
\geq sign\left( y-x\right) \Phi ^{^{\prime }}\left( \left\vert
y-x\right\vert \right) $ (see \cite[Theorem 6]{A1})\ similarly to case 
\textbf{a}, we get%
\begin{eqnarray*}
\left\vert x_{1,k}-\overline{x}\right\vert &=&\left\vert \sum_{i=1}^{k}\frac{%
t_{i}y_{i}}{\sum_{v=1}^{k}t_{v}}\right\vert \\
&\leq &\left( \frac{\sum_{i=1}^{n}t_{i}\left\vert y_{i}\right\vert ^{2r}}{%
\left( \sum_{i=1}^{k}t_{i}\right) +\left( \sum_{i=1}^{k}t_{i}\right)
^{2r}\left( \sum_{i=k+1}^{n}t_{i}\right) ^{1-2r}}\right) ^{\frac{1}{2r}} \\
. &\leq &\left( \frac{\sum_{i=1}^{n}t_{i}f\left( x_{i}\right) -f\left(
\sum_{i=1}^{n}\left( t_{i}x_{i}\right) \right) }{\left(
\sum_{i=1}^{k}t_{i}\right) +\left( \sum_{i=1}^{k}t_{i}\right) ^{2r}\left(
\sum_{i=k+1}^{n}t_{i}\right) ^{1-2r}}\right) ^{\frac{1}{2r}}
\end{eqnarray*}
\end{corollary}

We finish with an example:

\begin{example}
\label{Ex2} Choosing 
\begin{equation*}
t_{1}=\frac{1}{2},\quad t_{2}=-\frac{1}{2},\quad t_{3}=\frac{1}{2},\quad
t_{4}=\frac{1}{4},\quad t_{5}=-\frac{1}{4},\quad t_{6}=\frac{1}{2}
\end{equation*}%
and $k=3$. Also choosing $y_{i},$ $i=1,...,6$ as 
\begin{equation*}
y_{1}=-\frac{13}{2},\quad y_{2}=2,\quad y_{3}=3,\quad y_{4}=4,\quad
y_{5}=5,\quad y_{6}=6
\end{equation*}%
It is easy to verify that the conditions and therefore the results of
Theorem \ref{Th7} are satisfied.
\end{example}

\bigskip

\bigskip

\end{document}